\documentclass[12pt]{amsart}

\usepackage{amscd,amssymb, latexsym}
\usepackage{graphicx}
\usepackage{epstopdf}
\usepackage{amsmath,amsthm,amstext, amsbsy, amssymb}
\usepackage{mathrsfs}

\usepackage{enumerate}

\newtheorem{lmm}{Lemma}
\newtheorem{thr}{Theorem}
\newtheorem{stm}{Statement}
\newtheorem{crl}{Corollary}

\newtheorem{defn}{Definition}

\newcommand{\R}{\mathbb{R}}
\newcommand{\Compl}{\mathbb{C}}

\newcommand{\smm}{\displaystyle\sum}

\newcommand{\loctheta}{\tilde{\theta}}
\newcommand{\eps}{\varepsilon}
\newcommand{\mes}{\mu}
\newcommand{\low}{\varkappa}
\newcommand{\scprod}{\delta}
\newcommand{\spd}{\sigma}
\newcommand{\data}{\nu}

\begin{document}

\sloppy
\title[Jacobi matrices with lacunary spectrum]{Jacobi matrices with lacunary spectrum}
\author{Ilya Losev}

\address{
Ilya Losev:
\newline Department of Mathematics and Computer Science, St.~Petersburg State University, St. Petersburg, Russia
\newline {\tt Losevilya14@gmail.com}
}

\thanks{The work was supported by the Russian Science Foundation grant 19-11-00058.}

\begin{abstract}
We find asymptotics of entries of Jacobi matrices with lacunary spectral data under some additional growth conditions. We also prove the inverse results. In addition, we study connections between Jacobi matrices, canonical systems and de Branges spaces for lacunary spectral data.
\end{abstract}



\maketitle

\section{Introduction and main results}

One of the main problems in mathematical physics is to find correspondences between classes of potentials and spectral data for some systems of second order differential equations. We are studying connections between some classes of Jacobi matrices, de Branges spaces and canonical systems. The famous Krein -- de Branges theory deals with correspondence between de Branges spaces and canonical systems. On the other hand, every Jacobi matrix generates a canonical system of special type, i.e. canonical system with Hamiltonian consisting of one chain of indivisible intervals, see e.g.  \cite[Theorems 2,4]{RRomJM}, and \cite{deBrangesHilbSp, RRom} for general theory. Canonical systems, Krein -- de Branges theory, Jacobi matrices theory and connections between them have been intensively studied for the last 60 years \cite{ArovDym, GaussPr, KacKrein, Levitan, Simon}. For recent developments see, e.g. \cite{KWW, Remling, RemlingBook}.

\medskip

Very few one-to-one correspondences between Hamiltonian classes (canonical systems) and spectral data (de Branges spaces) are known. The main examples are de Branges solution of inverse spectral problem for arbitrary canonical system on finite interval \cite{deBrangesHilbSp} and Krein -- de Branges type formula \cite[Theorem 11]{RRom}. Another example is a localization phenomenon: de Branges space has a localization property (for zeros) if and only if the corresponding Hamiltonian (canonical system) consists of indivisible intervals accumulating only to the left \cite[Theorem 1.7]{ZeroLoc}. Recently R. Bessonov and S. Denisov found a beautiful description of Szeg\H{o} measures in terms of Hamiltonians \cite[Theorem 1]{BesDen}. R. Romanov was able to find the precise formula for the order of de Branges space generated by zero-diagonal Jacobi matrix \cite[Theorem 2]{RomValent}.

\medskip

The aim of this paper is to find some properties of Jacobi matrices or Hamiltonians (canonical systems) generated by so called \textit{small de Branges spaces}, i.e. such that spectra $\{t_n\}$ is lacunary, $t_{n+1}>\lambda t_n$ for some $\lambda>1$. Class of small de Branges spaces, introduced in \cite{DiscrHilb}, naturally appears in many topics of complex analysis such as spectral synthesis \cite[Theorems 1.1, 1.2]{SpectrSynth}, Riesz bases of reproducing kernels in Fock-type spaces \cite{FockRiesz} and even in Gabor analysis.
We are able to show that each small de Branges space (i.e. with lacunary spectral data) generates a Jacobi matrix with exponentially increasing entries (Theorem \ref{FTTheorem}). On the other hand, if Jacobi matrix has exponentially increasing entries then (under some technical assumptions) the spectral data is also lacunary (Theorems \ref{ThrJacobi}, \ref{ThrRevearsed}, see also Theorems \ref{ThrCanon}, \ref{ThrRevearsedCanon} as their versions for canonical systems). 

\subsection{Jacobi matrices and de Branges spaces}
We give a short overview of Jacobi matrices and de Branges spaces theory.
\subsubsection*{Jacobi matrices.} Consider a Jacobi matrix
\begin{equation*}
J=\begin{pmatrix}
  	q_1 & \rho_1 & 0   & 0  &\cdots\\
  	\rho_1 & q_2 & \rho_2 & 0 &\cdots\\
  	0   & \rho_2 & q_3 & \rho_3 &\cdots\\ 
  	\vdots  & \vdots & \vdots & \vdots  & \ddots\\
 	\end{pmatrix}, \qquad q_k\in\R, \, \rho_k>0.
\end{equation*}
Let 
\begin{equation*}
f(z) = \left\langle \left(J-z\right)^{-1}e_1, e_1 \right\rangle, \quad e_1 = (1, 0, 0 , 0, \ldots)^{\mathsf{T}}, \qquad f:\Compl^{+} \to \Compl^{+}
\end{equation*}
be the corresponding Herglotz function (here $\Compl^{+} = \{z \in \Compl | \, \Im z >0\}$). It is known that any Herglotz function can be represented in the following form
\begin{equation*}
az+b+
\int_{\R}\left(\frac{1}{t-z}-\frac{t}{1+t^2}\right)d\mes(t), \qquad 
\end{equation*}
where $a \geq 0$, $b\in \R$ and $\mes \geq 0$, $\int_{\R} \frac{d\mes}{1+t^2}<\infty$. In our case $a=0$, and
the corresponding measure $\mes$ is called \textit{spectral data} of the Jacobi matrix $J$.

We are interested in the connection between coefficients of Jacobi matrix $\{q_j, \rho_j\}$ and spectral data $\mes$. Our main goal is to describe some classes of spectral data for which asymptotics of the corresponding coefficients $\{q_j, \rho_j\}$ can be found.

This connection is given by Stieltjes algorithm. If we write $f(z)$ as a continued fraction, we will have 
\begin{equation*}
f(z) = -\frac{1}{\displaystyle
z-q_1-\frac{\rho_1^2}{\displaystyle 
z-q_2-\frac{\rho_2^2}{\displaystyle
z-q_3-\frac{\rho_3^2}{\displaystyle
\dots}}}}.
\end{equation*}
We will consider a case when the measure $\mes$ is discrete
\begin{equation*}
\mes = \sum_{k=1}^\infty \mu_k \delta_{t_k}, \qquad t_k \in \R, \, \mu_k>0.
\end{equation*}

It is known that Jacobi matrices correspond to canonical systems, which, in turn, can be described by two types of Herglotz functions $f(z)=\dfrac{C(z)}{A(z)}$ and $-\dfrac{B(z)}{A(z)}$ (entire functions $A, B, C$ are monodromy matrix entries, see Section \ref{SectCanonSyst} for the details). 
The measure corresponding to the first function is denoted by $\mu$, while the measure corresponding to the second one turns out to be measure of the form $\nu=\sum_{k=1}^\infty \nu_k \delta_{r_k}$. Hence,


\begin{equation}
\label{FTHBmeasure}
-\frac{B(z)}{A(z)}=p+\sum_{k=1}^{\infty}\data_k\left(\frac{1}{r_k-z}-\frac{r_k}{r_k^2+1}\right), \qquad p\geq 0.
\end{equation}

Note that $t_k=r_k$ are the zeros of the function $A(z)$, but for us it will be more convenient to denote them differently.

The measure $\nu$ contains important information about the corresponding de Branges space.

\subsubsection*{De Branges spaces}
We remind the basic notions of de Branges theory.

\begin{defn} We say that an entire function $E(z)$ is of Hermite--Biehler class, if  $|E(z)|>|E(\bar{z})|$ for all $z\in \Compl^+$.
\end{defn}
In particular, any Hermite-Biehler function does not have zeros in the upper half-plane.  

It can be shown that function $E(z) = A(z)+iB(z)$ is Hermite-Biehler. Note that since functions $A(z)$ and $B(z)$ are real-valued for $z\in\R$, they can be recovered from $E(z)$ by the following formulae

\begin{equation*}
A(z) = \frac{1}{2}(E(z)+\overline{E(\overline{z})}), \qquad B(z) = \frac{1}{2i}(E(z)-\overline{E(\overline{z})}).
\end{equation*}


It is well-known that any Hermite--Biehler function generates a Hilbert space of entire functions (de Branges space).

\begin{defn} The de Branges space $\mathcal{H}(E)$ corresponding to an Hermite--Biehler function $E(z)$ is the space of entire functions $F(z)$ such that both functions $F(z)/E(z)$ and $\overline{F(\overline{z})}/E(z)$ are in Hardy class $H_2(\Compl^+)$. The scalar product is defined by 
$$\left\langle F, G\right\rangle_{\mathcal{H}(E)}=\frac{1}{\pi}\int_{\R}F(t)\overline{G(t)}\frac{dt}{|E(t)|^2}.$$ 
\end{defn}

For example, if $E(z) = e^{-i\pi z}$, then the corresponding de Branges space is just the Paley-Wiener space.

\medskip

The connection between properties of the spectral measure $\nu$ and the corresponding de Branges spaces has attracted recent attention.
De Branges spaces corresponding to the lacunary sequence $\{r_k\}$  ($r_{k+1}>\lambda r_k$ for some $\lambda>1$) are well studied \cite{FockSp, DiscrHilb}, in particular one can describe Bessel sequences and Riesz bases of reproducing kernels for such spaces.
%
%
%
%

\label{IntroFTSpaces}


We will study small de Branges spaces with lacunary spectral data $\nu=\nu_{A,B}=\sum_k \nu_k\delta_{r_k}$ satisfying the following condition

\begin{equation} \label{1.2}
\sum_{k<n}\data_k + r_n^2 \sum_{k>n}\frac{\data_k}{r_k^2} \leq C\data_n.
\end{equation}

This class is a natural one, e.g. it was shown in \cite[Theorem 1.2]{SpectrSynth} that the de Branges space is of radial Fock-type space if and only if $\{r_k\}$ is lacunary and \eqref{1.2} holds.
On the other hand, this class appears in spectral synthesis problem.

\medskip

Inequality \eqref{1.2} is equivalent to the lacunarity of sequences $\left\{\data_k\right\}$, $\left\{\frac{\data_k}{r_k^2}\right\}$. 
From now on we assume that $\data_{k+1} > \low \data_k$, $\frac{\data_{k+1}}{r_{k+1}^2}<\theta\frac{\data_k}{r_k^2}$ for some $\low>1$ and $\theta < 1$.

\medskip

\begin{defn}
We say that measure $\mes = \sum_{k} \mu_k \delta_{t_k}$ is completely lacunary if the following hold
\begin{enumerate}[(i)]
\begin{item}
$t_{k+1}>\lambda t_k$,
\end{item}
\begin{item}
$\frac{\mu_{k+1}}{t_{k+1}} >\low \frac{\mu_k}{t_k}$,
\end{item}
\begin{item}
$\frac{\mu_{k+1}}{t_{k+1}^2} <\theta \frac{\mu_k}{t_k^2}$.
\end{item}
\end{enumerate}
\end{defn}

\subsection{Main results.}
Now we are ready to state main results of the paper.

Given sufficiently large $\lambda$, $\low$ and $\theta^{-1}$, we are able to find asymptotics of the corresponding Jacobi matrix entries.
\begin{thr}
\label{FTTheorem}
Let $\nu_k$ and $r_k$ be defined by \eqref{FTHBmeasure}.
Suppose that $r_{k+1}>\lambda r_k$, $\nu_{k+1}>\low \nu_k$ and $\frac{\nu_{k+1}}{r_{k+1}^2}<\theta \frac{\nu_k}{r_k^2},$ for some $\lambda>10^6,\, \low>10^6,\, \theta<10^{-6}$. Then
\begin{enumerate}[(i)]
\begin{item}
 $\left(1-\dfrac{100}{\low}\right)\left( r_{n} +\dfrac{\nu_n}{\nu_{n+1}}r_{n+1}\right)
<q_n
<\left(1+\dfrac{100}{\low}\right) \left(r_n+\dfrac{\nu_n}{\nu_{n+1}}r_{n+1}\right)$;
\end{item}
\begin{item} $\left(1-\dfrac{1000}{\low}-\dfrac{1000}{\lambda}\right)\dfrac{\nu_{n}}{\nu_{n+1}}r_{n+1}^2
<\rho_n^2
<\left(1+\dfrac{1000}{\low}+\dfrac{1000}{\lambda}\right)\dfrac{\nu_{n}}{\nu_{n+1}} r_{n+1}^2.$
\end{item}
\end{enumerate}
\end{thr}

We do not know if the result holds for arbitrary constants $\lambda$, $\low$ and $\theta^{-1}$ bigger than one.

\medskip

We also study the case when the lacunarity conditions are imposed on the measure $\mu$ instead of $\nu$. We assume the lacunarity of sequences $\{t_k\}$, $\left\{\frac{\mu_k}{t_k}\right\}$ and $\left\{\frac{\mu_k}{t^2_k}\right\}$ with parameters $\lambda$, $\low$ and $\theta$ respectively.

\medskip

Next we will consider finite dimensional case, i.e. when $\mes = \sum_{k=1}^N \mu_k \delta_{t_k}$. 

If $\lambda, \,\low$ and $\theta^{-1}$ are big enough, then we can solve inverse spectral problem and find asymptotics of Jacobi matrix entries $\{q_n\}$ and $\{\rho_n\}$ (Theorems \ref{ThrJacobi}, \ref{ThrRevearsed}). 

\begin{thr}
\label{ThrJacobi}
Let $\mu$ be a completely lacunary measure with big lacunarity parameters $\lambda>1000$,
$\low>20$, $\frac{10}{\lambda} < \theta < \frac{1}{100}$.
Then
\begin{enumerate}[(i)]
\item $\left(1-\lambda^{-1}\right) t_{N-n+1} < q_n < \left(1+3\low^{-1}\right) t_{N-n+1}$;
\item  $\frac{1}{10} \theta^{-1}t_{N-n}^2 < \rho^2_n
<10\low^{-1}t_{N-n+1}t_{N-n}$.
\end{enumerate}
\end{thr}

The inverse theorem (for the direct spectral problem) also holds up to some constants.
\begin{thr}
\label{ThrRevearsed}
Let $\lambda>1000$, $\frac{10}{\lambda} < \theta < \frac{1}{1000}$ and $\low > 100.$ Let also
\begin{enumerate}[(i)]
\item $q_{n}> 3\lambda q_{n+1}$;
\item $20 \theta^{-1} q_{n+1}^2 
< \rho_n^2 
<\frac{1}{20} \low^{-1} q_{n}q_{n+1}$.
\end{enumerate}
Then 
\begin{enumerate}[(i)]
\item $\left(1-\low^{-1}\right) q_{N-n+1} < t_n < \left(1+\lambda^{-1}\right) q_{N-n+1}$,
\item $\mu$ is a completely lacunary measure.
\end{enumerate}
\end{thr}

If we omit the lacunarity condition for $\left\{\frac{\mu_n}{t_n}\right\}$, then it seems that the analysis is more complicated, see Section \ref{CounterEx}.

Both Theorems \ref{ThrJacobi} and \ref{ThrRevearsed} can be rewritten in terms of corresponding Hamiltonians (canonical systems), see Section \ref{SectCanonSyst}.

\subsection{Canonical Systems}
\label{SectCanonSyst}
A canonical system is a differential equation of the form
\begin{equation*}
\begin{pmatrix}
0 & -1\\
1 & 0
\end{pmatrix} \frac{d}{dx}Y(x, z) = z\mathsf{H}(x)Y(x, z), \qquad x\in (0, L),
\end{equation*}
where $\mathsf{H}$ is a locally summable $2\times 2$ matrix-valued function on $(0, L)$ such that $\mathsf{H}(x)\geq 0$ a.e. Function $\mathsf{H}$ is called \textit{Hamiltonian} of the system. Changing the variable we can assume that $\mathrm{tr} \, \mathsf{H}(x) = 1$ a.e.
 
Interval $I$ is called an \textit{indivisible} interval, if there exists $e\in \R^2$ such that for any $x \in I$: 
\begin{equation}\label{indivDef}
\mathsf{H}(x) = \langle\cdot, e \rangle e 
\end{equation}
and there is no larger interval $I'$ such that \eqref{indivDef} holds for a.e. $x \in I'$.

It is known that there is a correspondence between Jacobi matrices and canonical systems consisting only of indivisible intervals (see, for example, \cite{RRom, RRomJM}).

Let $\mathsf{H}$ be a Hamiltonian which consists of finite number of indivisible intervals. Let $\mathsf{H}=\left\langle\cdot, e_k\right\rangle e_k$ on the $k$-th interval. Denote its length by $l_k$. Then $J$ is the corresponding Jacobi matrix if and only if
\begin{align}
\label{JacCanonDiag}
q_j &= \frac{1}{l_j}\left(
\frac{\left\langle e_j, e_{j+1}\right\rangle}{\left\langle e^{\bot}_j, e_{j+1}\right\rangle}+
\frac{\left\langle e_{j-1}, e_{j}\right\rangle}{\left\langle e^{\bot}_{j-1}, e_{j}\right\rangle}
\right), \qquad j\geq 2;\\
\label{JacCanonDiagFirst}
q_1 &= \frac{1}{l_1}\left(
\frac{\langle e_1, e_2\rangle}{\langle e_1^{\bot}, e_2\rangle}-\frac{e_1^{-}}{e_1^{+}} 
\right);\\
\label{JacCanonNonDiag}
\rho_j &= -\frac{1}{\sqrt{l_{j+1}l_j}\left\langle e_j^{\bot}, e_{j+1}\right\rangle}, \qquad j\geq 1.
\end{align}

Now let $M(x,z)= \begin{pmatrix}
A_x(z) & C_x(z)\\
B_x(z) & D_x(z)
\end{pmatrix}$ be the monodromy matrix of the canonical system, in other words,
\begin{equation*}
M(0,z) = \begin{pmatrix}
1 & 0\\
0 & 1
\end{pmatrix}, 
\qquad
\begin{pmatrix}
0 & -1\\
1 & 0
\end{pmatrix} \frac{d}{dx}M(x, z) = z\mathsf{H}(x)M(x, z), \qquad x\in (0, L).
\end{equation*}
It is easy to see that all functions $A_x(z), B_x(z), C_x(z), D_x(z)$ are real-valued when $z\in \R$.  
The corresponding Hermite--Biehler function is given by $E(z) = A_L(z)+iB_L(z)$. The spectral data of the corresponding Jacobi matrix is given by $\dfrac{C_L(z)}{A_L(z)}$, and the spectral data of the corresponding de Branges space is given by $-\dfrac{B_L(z)}{A_L(z)}$.

Denote $\scprod_k = -\langle e^{\bot}_{k}, e_{k+1}\rangle$.
Then Theorem \ref{ThrJacobi} can be rewritten in terms of canonical systems.
\begin{thr}
\label{ThrCanon}
Let $\mu$ be a completely lacunary measure with big lacunarity parameters 
$\lambda>1000$, 
$\low>100$,
$\frac{10}{\lambda}<\theta < \frac{1}{100}$.
We take $l_1=1000 q_1^{-1}$ and 
$e_1 = \begin{pmatrix}
1\\
0\\
\end{pmatrix}$.
Then for the corresponding canonical system we have $\frac{1}{1001}< \scprod_1 <\frac{1}{1000}$,
\begin{equation*}
\dfrac{\delta_n}{\delta_{n+1}}>\dfrac{1}{10}\low, \qquad \dfrac{l_{n+1}\scprod_{n+1}}{l_n\scprod_n}>\dfrac{1}{10}\lambda, \qquad \dfrac{l_{n+1}\scprod_{n+1}^2}{l_n\scprod_n^2}>\dfrac{1}{1000}\theta^{-1}.
\end{equation*}
\end{thr}

Theorem \ref{ThrRevearsed} can also be rewritten in the following way.
\begin{thr}
\label{ThrRevearsedCanon}
Let $e_1 = \begin{pmatrix}
1\\
0\\
\end{pmatrix}$, $\frac{1}{1001}< \scprod_1 <\frac{1}{1000}$ and
\begin{equation*}
\dfrac{\delta_n}{\delta_{n+1}}>100\low, \qquad \dfrac{l_{n+1}\scprod_{n+1}}{l_n\scprod_n}>100\lambda, \qquad \dfrac{l_{n+1}\scprod_{n+1}^2}{l_n\scprod_n^2}>100\theta^{-1}.
\end{equation*}
for some
$\lambda>1000$, $\low>100$, $\frac{10}{\lambda}<\theta<\frac{1}{100}$ and any $1\leq n\leq N-1$. 
Then the corresponding spectral data $\mu$ is a completely lacunary measure. 
\end{thr}

\noindent {\bf Organization of the paper and notations.} Theorem \ref{ThrJacobi} is deduced from Lemmas \ref{mainlemma} and \ref{lemmaPoleBound} in Section \ref{ProofThrJacobi}. Lemmas \ref{mainlemma} and \ref{lemmaPoleBound} are proved in Section \ref{ProofLemmas}.
Theorem \ref{ThrRevearsed} is proved in Section \ref{ProofThrReversed}.
Theorems \ref{ThrCanon} and \ref{ThrRevearsedCanon} are proved in Section \ref{ProofCanonTheorems}. Theorem \ref{FTTheorem} is proved in Section \ref{SectProofFT}.

We prove our results by direct analysis of each step of Stieltjes algorithm.

\medskip

Throughout this paper we write $\mu^{(l)}$ and $t^{(l)}$ to refer to the spectral data on the step $l$ of the Stieltjes algorithm.
We write $f \sim g$ when $f = O(g)$ and $g = O(f)$. 

\section{Proof of Theorem \ref{ThrJacobi}}
\label{ProofThrJacobi}
Consider a finite Jacobi matrix
\begin{equation}
J=\begin{pmatrix}
  	q_1 		& \rho_1		&\cdots 	& 0\\
  	\rho_1 	& q_2 		&\cdots 	& 0\\
  	\vdots  & \vdots  	& \ddots & \vdots\\
  	0		& 0			& \cdots	 & q_N\\
 	\end{pmatrix}, \qquad q_k\in\R, \, \rho_k>0.
\end{equation}

Then the corresponding Herglotz function $f(z)$ is a rational function of the form
\begin{equation}
\label{HerglotzDef}
f(z) = \sum_{k=1}^N \frac{\mu_k}{t_k-z}, \qquad \mu_k>0, \, \sum_{k=1}^{N}\mu_k=1, \quad t_k \in \R.
\end{equation}
Denote $J^{(n)}$ the matrix obtained by deleting $n-1$ first rows and columns from $J$:
\begin{equation}
J^{(n)}=\begin{pmatrix}
  	q_n & \rho_n & 0  &\cdots & 0\\
  	\rho_n & q_{n+1} & \rho_{n+1} &\cdots & 0\\
  	\vdots & \vdots & \vdots & \ddots & \vdots\\
  	0	& 0		 & 0		& \cdots	 & q_N\\
 	\end{pmatrix}.
\end{equation}
Let $f^{(n)}(z)$
be the corresponding functions. Let also
\begin{equation}
f^{(n)}(z) = \sum_{k=1}^{N-n+1} \frac{\mu^{(n)}_k}{t^{(n)}_k-z}, \qquad \mu^{(n)}_k>0, \, \sum_{k=1}^{N-n+1}\mu_k^{(n)}=1, \quad t^{(n)}_k \in \R.
\end{equation}
It is known that 
\begin{equation}
\label{recurrenceFormula}
-\frac{1}{f^{(n)}(z)} = z - q_n + \rho^2_n f^{(n+1)}(z).
\end{equation}
Comparing the asymptotics we get
\begin{equation}
\label{coeffFormulas}
q_n = \sum_{k=1}^{N-n+1}\mu_k^{(n)}t^{(n)}_k, \qquad \rho^2_n = \sum_{k=1}^{N-n+1}\mu_k^{(n)}\left(t^{(n)}_k\right)^2 - \left(\sum_{k=1}^{N-n+1}\mu_k^{(n)}t^{(n)}_k \right)^2.
\end{equation}

In order to prove Theorem \ref{ThrJacobi} we will need the following Lemmas.
\begin{lmm}
\label{mainlemma}
Let $\mu$ be a completely lacunary measure with big lacunarity parameters
$\lambda>1000$, 
$\low>10$,
$\frac{10}{\lambda}<\theta < \frac{1}{100}$.
Then for any $k\leq N-n-1$ we have
\begin{equation*}
\lambda<\dfrac{t_{k+1}^{(n)}}{t_k^{(n)}}, \qquad \dfrac{\mu_{k+1}}{\mu_k}<\dfrac{\mu_{k+1}^{(n)}}{\mu_k^{(n)}}, \qquad \dfrac{\mu_{k+1}^{(n)}}{\left(t_{k+1}^{(n)}\right)^2}<5\theta \dfrac{\mu_{k}^{(n)}}{\left(t_{k}^{(n)}\right)^2}.
\end{equation*}
\end{lmm}

\begin{lmm}
\label{lemmaPoleBound} Let $\mu$ be a completely lacunary measure with big lacunarity parameters $\lambda>1000$, $\low>10$,
$\frac{10}{\lambda}<\theta < \frac{1}{100}$. Then
2e have $t_k < t^{(n)}_k <\left(1+3\low^{-1}\right) t_k$ for $1<n\leq N$.
\end{lmm}

Now we will deduce Theorem \ref{ThrJacobi} from Lemmas \ref{mainlemma} and \ref{lemmaPoleBound}. 
\begin{proof}[Proof of the Theorem \ref{ThrJacobi}]

We use \eqref{coeffFormulas} and Lemma \ref{lemmaPoleBound}. Estimates of $q_n$ are obviously given by the main term $\mu_{N-n+1}^{(n)}t^{(n)}_{N-n+1}$. Hence, the required inequalities hold.
Now we estimate $\rho_n$. By Cauchy inequality we have
$\sum_{k=1}^{N-n} \mu_k^{(n)}\left(t_k^{(n)}\right)^2>\left(\sum_{k=1}^{N-n} \mu_k^{(n)}t_k^{(n)}\right)^2$.

Hence,
\begin{multline*}
\rho_n^2>\mu_{N-n+1}^{(n)}\left(t^{(n)}_{N-n+1}\right)^2 -\left(\mu_{N-n+1}^{(n)}t^{(n)}_{N-n+1}\right)^2 - 2\mu_{N-n+1}^{(n)}t^{(n)}_{N-n+1} \left(\sum_{k=1}^{N-n} \mu_k^{(n)}t_k^{(n)}\right)\\
> \mu_{N-n+1}^{(n)}\mu_{N-n}^{(n)}\left(t^{(n)}_{N-n+1}\right)^2(1-3\lambda^{-1})
>\frac{1}{10} \theta^{-1}t_{N-n}^2.
\end{multline*}
On the other hand,
\begin{equation*}
\rho_n^2 < \mu_{N-n+1}^{(n)}(1-\mu_{N-n+1}^{(n)})\left(t^{(n)}_{N-n+1}\right)^2+\sum_{k=1}^{N-n} \mu_k^{(n)}\left(t_k^{(n)}\right)^2<
10 \low^{-1} t_{N-n+1}t_{N-n}.
\end{equation*}
\end{proof}

\section{Proof of Lemmas \ref{mainlemma} and \ref{lemmaPoleBound}} 
\label{ProofLemmas}
We will study one step in the recurrence formula \eqref{recurrenceFormula}.
Suppose
\begin{equation}
\label{oneStepStieltjes}
-\left(\sum_{k=1}^N \frac{\mu_k}{t_k-z}\right)^{-1}=z-b+\sum_{k=1}^{N-1}\frac{w_k}{s_k-z}.
\end{equation}
Our main goal is to establish some connections between $\{\mu_k, t_k\}$ and $\{w_k, s_k\}$.
Clearly
\begin{equation}
\smm_{k=1}^N \dfrac{\mu_k}{t_k-s_n}=0 \label{root}; \qquad
\smm_{k=1}^N \dfrac{\mu_k}{(t_k-s_n)^2}=\frac{1}{w_n}. 
\end{equation}

Let
\begin{equation}
\label{Msum}
M^{(l)}_n = \sum_{k=1}^{n-l+1} \mu_k^{(l)}, \qquad M_n = M^{(0)}_n.
\end{equation}
We will assume that $\frac{\mu_{k+1}}{\mu_k} > \low \frac{t_{k+1}}{t_k}$ and $\frac{\mu_{k+1}}{t_{k+1}^2} < \theta \frac{\mu_{k}}{t_k^2}$ for some $\low > 10$ and 
some $\theta<1$.

\subsection{Root localization}
Now we estimate $s_n$.
\begin{stm}
\label{approx}
The following inequalities hold
$$\mu_n \left(\dfrac{\mu_n}{t_{n+1}-t_n}+\smm_{k=n+1}^N \dfrac{\mu_k}{t_k-t_n}\right)^{-1}<s_n-t_n<\mu_n \left(\smm_{k=n+1}^N \frac{\mu_k}{t_k-t_n}- \frac{M_{n-1}}{t_n-t_{n-1}}\right)^{-1}.$$
\end{stm}
\begin{proof}
From \eqref{root} it follows that

$$\frac{\mu_n}{s_n-t_n}>\smm_{k=n+1}^N \frac{\mu_k}{t_k-t_n}-\smm_{l=1}^{n-1} \frac{\mu_l}{t_n-t_l}.$$
This gives the right inequality.
%
On the other hand, from \eqref{root} we obtain
$$\frac{\mu_n(t_{n+1}-s_n)}{s_n-t_n}
<\mu_{n+1}+\smm_{k=n+2}^N \frac{\mu_k(t_{n+1}-t_n)}{t_k-t_n}.$$
So, we have
$$\mu_n\frac{t_{n+1}-t_n}{s_n-t_n}<\mu_n+\mu_{n+1}+\smm_{k=n+2}^N \frac{\mu_k(t_{n+1}-t_n)}{t_k-t_n}.$$
This gives the left inequality.
\end{proof}

%

Now we prove that, if on the first $K$ steps of the Stieltjes algorithm the weights grow fast enough, then even after $K+1$ step the poles will not change too much (compared to the initial data).

\begin{lmm}
\label{bound}
If $\dfrac{\mu_{n+1}^{(k)}}{\mu_n^{(k)}}>\hat{\low}\dfrac{t^{(k)}_{n+1}}{t^{(k)}_n}$ for $0\leq k \leq K$ and some $\hat{\low}>5$, then $\dfrac{t_n^{(K+1)}}{t_n^{(0)}}<1+\dfrac{2}{\hat{\low}}$.
\end{lmm}
\begin{proof}
From Statement \ref{approx} we deduce that for $0\leq p\leq K$ we have 
\begin{align*}
t^{(p+1)}_n-t^{(p)}_n
&<\mu_n^{(p)}\left(\dfrac{\mu_{N-p}^{(p)}}{t_{N-p}^{(p)}}- \dfrac{M^{(p)}_{n-1}}{t_n^{(p)}-t_{n-1}^{(p)}}\right)^{-1}.
\end{align*}
We also have $t_n^{(p)}-t_{n-1}^{(p)}>t_n^{(p)}\left(1-\frac{1}{\lambda}\right)$ and $M^{(p)}_{n-1}<\frac{\lambda}{\lambda-1}\mu_{n-1}^{(p)}$. 
Therefore,
\begin{equation*}
\frac{t^{(p+1)}_n-t^{(p)}_n}{t_n^{(p)}}
<\left(\dfrac{\mu_{N-p}^{(p)}t_n^{(p)}}{t_{N-p}^{(p)}\mu_n^{(p)}}-1\right)^{-1}
<\dfrac{1}{(\hat{\low}-1)^{N-p-n}}.
\end{equation*}
We get that
\begin{equation}
\label{slowgrowth}
\frac{t^{(p+1)}_n}{t_n^{(p)}}<1+\frac{1}{(\hat{\low}-1)^{N-p-n}},
\end{equation}
from which Lemma \ref{bound} follows.
\end{proof}

\subsection{Invariance of the lacunary parameter}
We are going to show that lacunarity parameter of the poles $\{t_n\}$ does not change too much.

\begin{stm}
\label{lacun}
We have $s_{n+1}-t_{n+1}>\frac{\mu_{n+1}(\lambda-1)}{\mu_{n}\lambda}(s_n-t_n)$. In particular, $s_{n+1}-t_{n+1}>\lambda(s_n-t_n)$ and $s_{n+1}>\lambda s_n$.
\end{stm}
\begin{proof}
Using Statement \ref{approx} we see that it is sufficient to show that
$${\dfrac{\mu_{n+1}}{t_{n+2}-t_{n+1}}+\smm_{k=n+2}^N \dfrac{\mu_k}{t_k-t_{n+1}}}
<\dfrac{\lambda}{\lambda-1} \left({\smm_{k=n+1}^N \frac{\mu_k}{t_k-t_n}- \frac{M_{n-1}}{t_n-t_{n-1}}}\right)
.$$
We note that
$\frac{t_k-t_n}{t_k-t_{n+1}}<
\frac{\lambda}{\lambda-1}$ for $k\geq n+2$.
Therefore, it reduces to obvious inequality
\begin{equation*}
\frac{\mu_{n+1}}{t_{n+2}-t_{n+1}}<\frac{\lambda}{\lambda-1}\left(\frac{\mu_{n+1}}{t_{n+1}-t_n}-\frac{M_{n-1}}{t_n-t_{n-1}}\right).
\end{equation*}
\end{proof}

\subsection{Lower bound}
In this subsection we prove that under some assumptions the sequence $\frac{\mu_{n+1}}{\mu_n}$ does not decrease much after one step of Stieltjes algorithm.

\begin{stm}
\label{lowstep}
Let $\theta < \frac{1}{10}$.
Then $\frac{w_{n+1}}{w_n}> \frac{\mu_{n+1}}{\mu_n}$.
\end{stm}
\begin{proof}
From \eqref{root} we know that
$$\left(\smm_{k=1}^N \dfrac{\mu_k}{(t_k-s_{n+1})^2}\right)\frac{w_{n+1}}{w_n}=\smm_{k=1}^N \dfrac{\mu_k}{(t_k-s_n)^2}.$$
Let $k<n$. 
Then from 
$\frac{(s_{n+1}-t_k)^2}{(s_n-t_k)^2}>\frac{s_{n+1}^2}{s_n^2}>\frac{\mu_{n+1}}{2\theta\mu_n}$ it 
follows that $\frac{\mu_k}{(s_n-t_k)^2}>\frac{\mu_{n+1}}{\mu_n} \cdot \frac{\mu_k}{(s_{n+1}-t_k)^2}$. Also $\frac{\mu_{n+1}}{(t_{n+1}-s_n)^2}>\frac{\mu_{n+1}}{\mu_n} \cdot \frac{\mu_{n}}{(s_{n+1}-t_n)^2}.$
Therefore, it is sufficient to show that
\begin{equation}
\label{goal}
\frac{\mu_n^2}{(s_n-t_n)^2}
>\dfrac{\mu_{n+1}^2}{(s_{n+1}-t_{n+1})^2}+\smm_{k=n+2}^N \dfrac{\mu_{n+1}\mu_k}{(t_k-s_{n+1})^2}.
\end{equation}
The main asymptotics on the right-hand side is given by $\frac{\mu_{n+1}^2}{(s_{n+1}-t_{n+1})^2}$.

\vspace{3pt}
\textbf{Step 1. Estimate of the main term.} By subtracting \eqref{root} for $n$ from \eqref{root} for $n+1$ we get

\begin{equation}
\label{mainEqality}
\dfrac{\mu_{n+1}}{s_{n+1}-t_{n+1}}
=\dfrac{\mu_n}{s_n-t_n}\cdot\dfrac{t_{n+1}-s_n}{s_{n+1}-t_{n}}+\smm_{\substack{1 \leq k \leq N\\
k\neq n, n+1}} \dfrac{\mu_k(t_{n+1}-s_n)}{(s_{n+1}-t_k)(s_n-t_k)}.
\end{equation}
Hence,
\begin{multline}
\label{keyineq}
\dfrac{\mu_{n+1}}{s_{n+1}-t_{n+1}}<
\dfrac{\mu_n}{s_n-t_n}-\dfrac{\mu_n}{s_{n+1}-t_{n}}\cdot\left(1+\dfrac{s_{n+1}-t_{n+1}}{s_n-t_n}\right)\\
+\dfrac{M_{n-1}(t_{n+1}-t_n)}{(t_{n+1}-t_{n-1})(t_n-t_{n-1})}+\smm_{l=n+2}^N \dfrac{\mu_l(t_{n+1}-t_n)}{(t_l-s_{n+1})^2}\\
=\mathfrak{I_1}-\mathfrak{I_2}+\mathfrak{I_3}+\mathfrak{I_4}.
\end{multline}

Now we need to estimate $\mathfrak{I_2}, \mathfrak{I_3}$ and $\mathfrak{I_4}$. 

\textbf{Step 2.}  We are going to prove that $\mathfrak{I_2}>2(\mathfrak{I_3}+\mathfrak{I_4})$.
From Statement \ref{lacun}, we see that it is sufficient to prove two inequalities:
\begin{equation}
\label{left}
\frac{1}{\lambda}\cdot\frac{\mu_n}{s_{n+1}-t_n}\cdot \dfrac{\mu_{n+1}(\lambda-1)}{\mu_n\lambda}
>\dfrac{\sum_{k=1}^{n-1}\mu_k(t_{n+1}-t_n)}{(t_{n+1}-t_{n-1})(t_n-t_{n-1})},
\end{equation}
\begin{equation}
\label{right}
\frac{1}{3}\cdot\frac{\mu_n}{s_{n+1}-t_n}\cdot\dfrac{\mu_{n+1}(\lambda-1)}{\mu_n\lambda}
>\smm_{l=n+2}^N \dfrac{\mu_l(t_{n+1}-t_n)}{(t_l-s_{n+1})^2}.
\end{equation}

\textbf{2a.} Now we are going to prove \eqref{left}. Clearly $\mathfrak{I_4} < \left(\frac{\lambda}{\lambda-1}\right)^3\cdot\frac{\mu_{n-1}}{t_n}$.
%
Therefore, it reduces to obvious inequality
\begin{equation*}
\frac{\mu_{n+1}}{s_{n+1}-t_n}>\lambda\left(\dfrac{\lambda}{\lambda-1}\right)^4\cdot\dfrac{\mu_{n-1}}{t_n}.
\end{equation*}

\textbf{2b.} Now we are going to prove \eqref{right}.
Note that
$$\dfrac{\mu_{l+1}}{(t_{l+1}-s_{n+1})^2}\left(\dfrac{\mu_l}{(t_l-s_{n+1})^2}\right)^{-1}
<\dfrac{\mu_{l+1}}{\mu_l}\cdot\frac{t_l^2}{t_{l+1}^2}
<\theta.$$
Hence, using Lemma \ref{bound}, we get
%
\begin{gather*}
\smm_{l=n+2}^N \dfrac{\mu_l(t_{n+1}-t_n)}{(t_l-s_{n+1})^2}
<\dfrac{1}{(1-\theta)(1-\frac{3}{\lambda})^2}\cdot\dfrac{\mu_{n+2}t_{n+1}}{t_{n+2}^2}.
\end{gather*}
In order to prove inequality \eqref{right} it remains to show that
$$\dfrac{\mu_{n+1}}{t_{n+1}^2}>6\dfrac{\lambda}{(\lambda-1)(1-\theta)(1-\frac{3}{\lambda})^2}\dfrac{\mu_{n+2}}{t_{n+2}^2},$$
which is true for $\theta<\frac{1}{10}$.
So, we have proved inequalities \eqref{left} and \eqref{right}, and, therefore, $\mathfrak{I_2}>2(\mathfrak{I_3}+\mathfrak{I_4})$. Hence, because of \eqref{keyineq}, we have
\begin{equation}
\label{halfway}
\dfrac{\mu_{n+1}}{s_{n+1}-t_{n+1}}<
\dfrac{\mu_n}{s_n-t_n}-\frac{1}{2}\cdot\dfrac{\mu_n}{s_{n+1}-t_{n}}\cdot\left(1+\dfrac{s_{n+1}-t_{n+1}}{s_n-t_n}\right).
\end{equation}

\textbf{Step 3.}
We want to prove inequality \eqref{goal}. Rewrite it in the following way
\begin{equation}
\left(\frac{\mu_n}{s_n-t_n}-\dfrac{\mu_{n+1}}{s_{n+1}-t_{n+1}}\right)\left(\frac{\mu_n}{s_n-t_n}+\dfrac{\mu_{n+1}}{s_{n+1}-t_{n+1}}\right)
>\smm_{k=n+2}^N \dfrac{\mu_{n+1}\mu_k}{(t_k-s_{n+1})^2}.
\end{equation}
From \eqref{halfway} it follows that it is sufficient to show that
\begin{equation*}
\dfrac{\mu_n}{s_{n+1}-t_{n}}\cdot\frac{\mu_{n+1}}{\mu_n}\cdot\dfrac{\mu_{n+1}}{s_{n+1}-t_{n+1}}
>\dfrac{1}{1-\theta}\cdot\dfrac{\mu_{n+1}\mu_{n+2}}{(1-\frac{3}{\lambda})^3(t_{n+2}-s_{n+1})^2},
\end{equation*}
which holds because of inequality \eqref{slowgrowth} and $\theta<\frac{1}{10}$.

\end{proof}

\subsection{Upper bound}
We are going to show how do the ratios of successive elements of  $\left\{\frac{\mu_n}{t_n^2}\right\}_{n=1}^N$ change after one step of the Stieltjes algorithm.
\begin{stm}
\label{upstep}
Let $\loctheta = \max \left\{\frac{\mu_{n+1}t_{n}^2}{\mu_{n}t_{n+1}^2}, \frac{1}{\lambda}\right\}$.
Then $\frac{w_{n+1}s_{n}^2}{w_{n}s_{n+1}^2}<\left(1+\eps_n\right)^3 \loctheta$, where $\eps_n = 5(\low-1)^{n-N}$.
\end{stm}
\begin{proof}
We want to prove the following inequality
\begin{equation}
\label{goallessstrong}
\smm_{k=1}^N \dfrac{\mu_ks_n^2}{(t_k-s_n)^2}<\left(1+\eps_n\right)^3\loctheta\smm_{k=1}^N \dfrac{\mu_ks_{n+1}^2}{(t_k-s_{n+1})^2}.
\end{equation}

\textbf{Step 1.} First, we will compare the main terms on both sides of the inequality \eqref{goallessstrong}.
We are going to show that
\begin{equation}
\label{mainlessstrong}
\dfrac{\mu_n s_n^2}{(s_n-t_n)^2}<\left(1+\eps_n\right)^2\loctheta \dfrac{\mu_{n+1}s_{n+1}^2}{(s_{n+1}-t_{n+1})^2}.
\end{equation}
It is sufficient to prove that
\begin{equation*}
\dfrac{\mu_n^2 s_n^2}{t_n^2(s_n-t_n)^2}<\left(1+\eps_n\right)^2\dfrac{\mu_{n+1}^2s_{n+1}^2}{t_{n+1}^2(s_{n+1}-t_{n+1})^2},
\end{equation*}
which will follow from
\begin{equation*}
\dfrac{\mu_n}{s_n-t_n}-\dfrac{\mu_{n+1}}{s_{n+1}-t_{n+1}}<\dfrac{\mu_{n+1}}{t_{n+1}}-\dfrac{\mu_n}{t_n}+\eps_n\cdot\dfrac{\mu_{n+1}}{s_{n+1}-t_{n+1}}.
\end{equation*}

Notice that from the identity \eqref{root} we can get that
\begin{equation} \label{upperRootDiffer}
\frac{\mu_n}{s_n-t_n}-\frac{\mu_{n+1}}{s_{n+1}-t_{n+1}} = \frac{\mu_n}{s_{n+1}-t_n}-\frac{\mu_{n+1}}{s_n-t_{n+1}}+\smm_{k\neq n, n+1}\frac{\mu_k(s_n-s_{n+1})}{(s_{n+1}-t_k)(s_n-t_k)}.
\end{equation}
We need to prove that the right-hand side is smaller than
$\dfrac{\mu_{n+1}}{t_{n+1}}-\dfrac{\mu_n}{t_n}+\eps_n\cdot\dfrac{\mu_{n+1}}{s_{n+1}-t_{n+1}}$. This reduces to inequality
\begin{equation*}
\frac{\mu_ns_{n+1}}{(s_{n+1}-t_n)t_n}+\frac{\mu_{n+1}s_n}{(t_{n+1}-s_n)t_{n+1}}
<\eps_n\cdot\dfrac{\mu_{n+1}}{s_{n+1}-t_{n+1}},
\end{equation*}
%
which follows from \eqref{slowgrowth}.
Therefore, we have proved inequality \eqref{mainlessstrong}.  
\vspace{1pt}

\textbf{Step 2.} In order to prove \eqref{goallessstrong} it is sufficient to show that
\begin{equation}
\label{localgoalweak}
\smm_{k\neq n} \dfrac{\mu_ks_n^2}{(t_k-s_n)^2}<\tilde{\theta}\left(\dfrac{\mu_{n}s_{n+1}^2}{(s_{n+1}-t_n)^2}+\eps_n\dfrac{\mu_{n+1}s_{n+1}^2}{(s_{n+1}-t_{n+1})^2}\right).
\end{equation}
We will prove that
\begin{equation}
\label{oneweak}
\smm_{k=1}^{n-1} \dfrac{\mu_ks_n^2}{(t_k-s_n)^2}
<\tilde{\theta}\dfrac{\mu_{n}s_{n+1}^2}{(s_{n+1}-t_n)^2};
\end{equation}
\begin{equation}
\label{twoweak}
\smm_{k=n+1}^{N} \dfrac{\mu_ks_n^2}{(t_k-s_n)^2}
<\tilde{\theta}\eps_n\dfrac{\mu_{n+1}s_{n+1}^2}{(s_{n+1}-t_{n+1})^2}.
\end{equation}
From these two inequalities \eqref{localgoalweak} follows.

\textbf{2a.}
Notice that
\begin{equation*}
\tilde{\theta}\frac{\mu_{n}s_{n+1}^2}{(s_{n+1}-t_n)^2}>\tilde{\theta} \mu_n>2\mu_{n-1}>\smm_{k=1}^{n-1} \dfrac{\mu_ks_n^2}{(t_k-s_n)^2}.
\end{equation*}
This proves \eqref{oneweak}.

\textbf{2b.} Inequality \eqref{twoweak} follows from \eqref{slowgrowth}, $\loctheta>\frac{1}{\lambda}$ and 
\begin{equation*}
\smm_{k=n+1}^{N} \dfrac{\mu_ks_n^2}{(t_k-s_n)^2}
<\frac{\mu_{n+1}s_n^2}{(1-\theta)(1-\frac{3}{\lambda})^2t_{n+1}^2}.
\end{equation*}

We have proved \eqref{twoweak} and, therefore, \eqref{localgoalweak}.
\end{proof}
Now Lemmas \ref{mainlemma} and \ref{lemmaPoleBound} follow by induction from Statements \ref{lacun}, \ref{lowstep}, \ref{upstep} and Lemma \ref{bound}.

\section{Proof of Theorem \ref{ThrRevearsed}}
\label{ProofThrReversed}
We will analyse one step of reversed Stieltjes algorithm. 
Suppose equation \eqref{oneStepStieltjes} holds. Then obviously
\begin{equation}
\label{rootReversed}
t_n - b + \sum_{k=1}^{N-1}\frac{w_k}{s_k-t_n} = 0; \qquad
1 + \sum_{k=1}^{N-1}\frac{w_k}{(s_k-t_n)^2} = \frac{1}{\mu_n}.
\end{equation}
We will assume that 
\begin{enumerate}[(i)]
\item $b > s_{N-1}\lambda$ and $s_{n+1}>\lambda s_n$ for any $n$;
\item $\frac{w_{n+1}}{w_n} > \low \frac{s_{n+1}}{s_n}$ and $\frac{w_{n+1}}{s_{n+1}^2}<\theta\frac{w_n}{s_n^2}$ for any $n$ and some $\low > 10$, $\theta<1$;
\item $10\theta^{-1} s_{N-1}^2 < \sum_{k=1}^{N-1}w_k < \frac{1}{10} \low^{-1} b s_{N-1}$.
\end{enumerate}
We will proceed in the same way as in Section \ref{ProofLemmas}.

\subsection{Root localization}
We are going to estimate $t_n$.
\begin{stm}
\label{DirectRootLocal}
We have
$\left(1-\frac{2}{\low}\right)\frac{w_n}{b}<s_n-t_n<\left(1+\frac{2}{\lambda}+\frac{2}{\low}\right) \frac{w_n}{b}$ for $n<N$ and $b<t_N<\left(1+\frac{1}{\lambda\low}\right) b$.
\end{stm}
\begin{proof}
\textbf{Case 1. $n<N.$} From \eqref{rootReversed} we have
\begin{equation*}
\frac{w_n}{s_n-t_n} 
> b\left(1-\frac{1}{\lambda}\right)-\left(1+\frac{2}{\lambda}\right)\sum_{k>n}\frac{w_k}{s_k}
> b\left(1-\frac{1}{\lambda}\right)-\frac{\low^{-1}bs_{N-1}}{5s_{N-1}}
>b \left(1-\frac{1}{\lambda}-\frac{1}{\low}\right).
\end{equation*} 
Hence, $s_n-t_n<\left(1+\frac{2}{\lambda}+\frac{2}{\low}\right)\frac{w_n}{b}.$ 
On the other hand,
\begin{equation*}
\frac{w_n}{s_n-t_n}
< b+\left(1+\frac{2}{\lambda}\right)\sum_{k<n}\frac{w_k}{t_n}
< b + \frac{\low^{-1}bs_{N-1}}{5s_{N-1}}
< \left(1+\frac{1}{\low}\right) b.
\end{equation*}

\textbf{Case 2. $n=N.$}
From \eqref{rootReversed} we get $t_N>b$ and
\begin{equation*}
t_N<b+\sum_{k=1}^{N-1}\frac{w_k}{b-s_k}<b\left(1+\frac{1}{\lambda\low}\right).
\end{equation*}
\end{proof}
\begin{crl}
\label{reversedRootCrl}
For $n<N$ we have $s_n-t_n<\frac{1}{5} \low^{n-N} s_n$.
\end{crl}
\begin{proof}
It follows from  $\frac{w_n}{bs_n} < \low^{n-N+1}\frac{w_{N-1}}{bs_{N-1}}$.
\end{proof}

\subsection{Lower bound}
Now we are going to estimate $\frac{\mu_{n+1}}{\mu_n}$.
\begin{stm}
We have $(1+\eps_n)^2\frac{\mu_{n+1}}{\mu_n}>\frac{w_{n+1}}{w_n}$, where $\eps_n = \low^{n-N+1}$, and $\frac{\mu_{N}}{\mu_{N-1}}>\low \frac{t_N}{t_{N-1}}.$
\end{stm}

\begin{proof}

\textbf{Case 1. $n<N-1.$}
It is sufficient to show that 
\begin{equation*}
(1+\eps_n)^2\frac{w^2_n}{(s_n-t_n)^2}
> \left(w_{n+1} + \sum_{k=1}^{N-1}\frac{w_kw_{n+1}}{(s_k-t_{n+1})^2}\right).
\end{equation*}

We are going to show that
\begin{equation*}
\left(\frac{(1+\eps_n)w_n}{s_n-t_n}-\frac{w_{n+1}}{s_{n+1}-t_{n+1}}\right)
\left(\frac{(1+\eps_n)w_n}{s_n-t_n}+\frac{w_{n+1}}{s_{n+1}-t_{n+1}}\right)
> w_{n+1} + \sum_{\substack{1\leq k\leq N-1\\
k\neq n+1}}\frac{w_kw_{n+1}}{(s_k-t_{n+1})^2}.
\end{equation*}

Analogously to \eqref{mainEqality} we have
\begin{equation*}
\frac{w_{n}}{s_n-t_n}
= \frac{w_{n+1}}{s_{n+1}-t_{n+1}}\cdot \frac{t_{n+1}-s_{n}}{s_{n+1}-t_n}
+t_{n+1}-s_{n}
+ \sum_{k\neq n, n+1}\frac{w_k(t_{n+1}-s_n)}{(t_{n+1}-s_k)(t_n-s_k)}.
\end{equation*}

From Corollary \ref{reversedRootCrl} we have $\frac{t_{n+1}-s_n}{s_{n+1}-t_n}>1-\low^{n-N+1}$.
Hence, 
\begin{equation*}
(1+\eps_n)\frac{w_n}{s_n-t_n}-\frac{w_{n+1}}{s_{n+1}-t_{n+1}}
> \left(1- \frac{3^{n-N+1}}{2\lambda}\right)t_{n+1}
\left(1+\sum_{k\neq n, n+1}\frac{w_k(t_{n+1}-s_n)}{(t_{n+1}-s_k)(t_n-s_k)}\right).
\end{equation*}
Also, 
\begin{equation*}
(1+\eps_n)\frac{w_n}{s_n-t_n}+\frac{w_{n+1}}{s_{n+1}-t_{n+1}}>b.
\end{equation*}
Moreover,
$bt_{n+1}>5w_{n+1}$. This gives the desired.

\textbf{Case 2. $n=N-1.$} It is sufficient to show that
%
\begin{equation*}
\frac{w_{N-1}^2}{(s_{N-1}-t_{N-1})^2}
> 2\low \frac{b}{s_{N-1}}\left(w_{N-1}+\sum_{k=1}^{N-1}\frac{w_{N-1}w_k}{(s_k-t_{N})^2}\right).
\end{equation*}

Analogously to \eqref{mainEqality} we have
\begin{equation*}
\frac{w_{N-1}}{s_{N-1}-t_{N-1}}
= t_{N}-s_{N-1}
+ \sum_{k=1}^{N-2}\frac{w_k(t_{N}-s_{N-1})}{(t_{N}-s_k)(t_{N-1}-s_k)}.
\end{equation*}
Also, 
$\frac{w_{N-1}}{s_{N-1}-t_{N-1}}>\frac{b}{2}.$ Moreover,
\begin{equation*}
\frac{w_{N-1}^2}{(s_{N-1}-t_{N-1})^2}
> 5 \low\frac{b}{s_{N-1}}\cdot \frac{w_{N-1}^2}{(t_N-s_{N-1})^2}.
\end{equation*}
Also, $bt_N > 10\low \dfrac{b}{s_{N-1}} w_{N-1}.$ This gives the required inequality.
\end{proof}

\subsection{Upper bound}
Now we are going to estimate $\frac{\mu_n}{t_n^2}$.

\begin{stm}

Let $\theta<\frac{1}{10}$. Denote  $\tilde{\theta} = \max\left\{\frac{w_{n+1}s_n^2}{w_ns_{n+1}^2}, \frac{1}{\lambda}\right\}$. 
Then $\left(1+\eps_n\right)^3\tilde{\theta}>\frac{\mu_{n+1}t_{n}^2}{\mu_{n}t_{n+1}^2}$ for $\eps_n = 6\lambda^{n-N+1}+2\low^{n-N+1}$ and $\frac{\mu_N}{t_N^2}<\theta \frac{\mu_{N-1}}{t_{N-1}^2}$.
\end{stm}
\begin{proof}
\textbf{Case 1.} $n < N-1$. 

It is sufficient to prove that
\begin{equation}
t_n^2 + \sum_{k=1}^{N-1}\frac{t_n^2w_k}{(s_k-t_n)^2} 
< (1+\eps_n)^3 \tilde{\theta}\left( t_{n+1}^2 + \sum_{k=1}^{N-1}\frac{t_{n+1}^2w_k}{(s_k-t_{n+1})^2} \right).
\end{equation}

\textbf{Step 1.} First, we compare main terms. We are going to prove that
\begin{equation} \label{revearsedUppStepOneMain}
\frac{t_n^2w_n}{(s_n-t_n)^2} 
< (1+\eps_n)^2\tilde{\theta}\frac{t_{n+1}^2w_{n+1}}{(s_{n+1}-t_{n+1})^2}.
\end{equation}
This will follow from 
\begin{equation*}
\frac{w_n}{s_n-t_n} - \frac{w_{n+1}}{s_{n+1}-t_{n+1}}
<  \frac{w_n}{s_n} - \frac{w_{n+1}}{s_{n+1}} + \frac{\eps_n}{2} \cdot \frac{w_{n+1}}{s_{n+1}-t_{n+1}}.
\end{equation*}
Notice that from \eqref{rootReversed} we have
\begin{equation*}
\frac{w_n}{s_n-t_n} - \frac{w_{n+1}}{s_{n+1}-t_{n+1}} 
= t_{n+1}-t_n + \frac{w_n}{s_n-t_{n+1}} - \frac{w_{n+1}}{s_{n+1}-t_n} + \sum_{k \neq n, n+1} \frac{w_k(t_{n+1}-t_n)}{(s_k-t_{n+1})(s_k-t_{n})}.
\end{equation*}
So, we want to prove that
\begin{equation*}
\frac{\eps_n}{2} \cdot \frac{w_{n+1}}{s_{n+1}-t_{n+1}}
> t_{n+1}-t_n - \frac{w_nt_{n+1}}{s_n(t_{n+1}-s_n)} - \frac{w_{n+1}t_n}{s_{n+1}(s_{n+1}-t_n)} + \sum_{k \neq n, n+1} \frac{w_k(t_{n+1}-t_n)}{(s_k-t_{n+1})(s_k-t_{n})}.
\end{equation*}
Note that 
\begin{equation*}
\sum_{k \neq n, n+1} \frac{w_k(t_{n+1}-t_n)}{(s_k-t_{n+1})(s_k-t_{n})}
<2\frac{w_{n-1}}{s_n} + 3\frac{w_{n+2}s_{n+1}}{s_{n+2}^2}
\end{equation*}
We have
\begin{equation*}
2\frac{w_{n-1}}{s_n} < \frac{w_nt_{n+1}}{s_n(t_{n+1}-s_n)},
\end{equation*}
and
\begin{equation*}
3\frac{w_{n+2}s_{n+1}}{s_{n+2}^2}
< 3\frac{w_{n+1}}{s_{n+1}} < \frac{3}{5} \low^{n-N+1}\frac{w_{n+1}}{s_{n+1}-t_{n+1}}.
\end{equation*}
Also,
\begin{equation*}
t_{n+1}-t_n 
< 2\lambda^{n-N+1} b < 3 \lambda^{n-N+1} \frac{w_{n+1}}{s_{n+1}-t_{n+1}}.
\end{equation*}
This gives us \eqref{revearsedUppStepOneMain}.

\textbf{Step 2.} Now we are going to compare the remaining terms on both sides. We will show that
\begin{equation*}
t_n^2 + \sum_{k\neq n}\frac{t_n^2w_k}{(s_k-t_n)^2} 
< \tilde{\theta}\left( t_{n+1}^2 + \frac{t_{n+1}^2w_n}{(t_{n+1}-s_n)^2} 
+ \eps_n\frac{t^2_{n+1}w_{n+1}}{(s_{n+1}-t_{n+1})^2}\right).
\end{equation*}
Indeed,
\begin{gather*}
t_n^2 < \tilde{\theta}t_{n+1}^2, \\
\sum_{k < n}\frac{t_n^2w_k}{(s_k-t_n)^2} 
< 2 w_{n-1}< \tilde{\theta} \frac{t_{n+1}^2w_n}{(t_{n+1}-s_n)^2}, \\
\sum_{k > n}\frac{t_n^2w_k}{(s_k-t_n)^2}
< \frac{3t_n^2w_{n+1}}{(s_{n+1}-t_n)^2}
< \low^{n-N+1}\tilde{\theta}\frac{t^2_{n+1}w_{n+1}}{(s_{n+1}-t_{n+1})^2}.
\end{gather*}

%
%

\textbf{Case 2.} $n = N-1$.

\textbf{Step 1.} We again compare main terms. Since $s_{N-1}-t_{N-1}>\frac{w_{N-1}}{2b}$,
\begin{equation*}
\frac{t_{N-1}^2w_{N-1}}{(s_{N-1}-t_{N-1})^2} < \frac{\theta}{2} t_{N}^2.
\end{equation*}

\textbf{Step 2.} We show that
\begin{equation*}
t_{N-1}^2 + \sum_{k \leq N-2}\frac{t_{N-1}^2w_k}{(s_k-t_{N-1})^2} 
< \frac{\theta}{2}\left( t_{N}^2 + \frac{t_{N}^2w_{N-1}}{(t_{N}-s_{N-1})^2}\right).
\end{equation*}
Indeed,
\begin{gather*}
t_{N-1}^2 < \frac{\theta}{2}t_{N}^2, \\
\sum_{k \leq N-2}\frac{t_{N-1}^2w_k}{(s_k-t_{N-1})^2} 
< 2 w_{N-2}< \frac{\theta}{2} \frac{t_{N}^2w_{N-1}}{(t_{N}-s_{N-1})^2}.
\end{gather*}
\end{proof}

Now Theorem \ref{ThrRevearsed} follows by induction.

\subsection{Remark about lacunarity of $\left\{\frac{\mu_k}{t_k}\right\}$}
\label{CounterEx}
We show that if we omit the lacunarity condition for $\left\{\frac{\mu_{k}}{t_k}\right\}$, then analysis of the corresponding Jacobi matrix coefficients can be more cumbersome. We still assume lacunarity of $\{t_k\}$, $\{\mu_k\}$ and $\left\{\frac{\mu_k}{t_k^2} \right\}$.

We fix some spectral data $\left\{\mu_k\right\}_{k=1}^N$, $\left\{t_k\right\}_{k=1}^N$ that satisfies these conditions and then let $t_N$ tend to $\infty$. Now we analyze what happens after one step of the Stieltjes algorithm (we again use \eqref{oneStepStieltjes}). We will study asymptotics of $q_n.$

Let also
\begin{equation*}
-\left(\sum_{k=1}^{N-1}\frac{\mu_k}{t_k-z}\right)^{-1} = a'z-b'+\sum_{k=1}^{N-2}\frac{w_k'}{s_k'-z}.
\end{equation*}
Then, because of \eqref{root}, when $t_N \to \infty$ we have $s_k \to s'_k$ and $w_k \to w'_k$ for $1\leq k \leq N-2$. Now we need to analyze asymptotics of $s_{N-1}$ and $w_{N-1}$. 

We have
\begin{equation*}
\sum_{k=1}^{N-1} \frac{\mu_k}{s_{N-1}-t_k} = \frac{\mu_N}{t_N-s_{N-1}}.
\end{equation*}
Hence, 
\begin{equation*}
\frac{s_{N-1}}{t_N} > \frac{\sum_{k=1}^{N-1}\mu_k}{\sum_{k=1}^{N}\mu_k}, \qquad
\frac{t_N-s_{N-1}}{t_N-t_{N-1}} > \frac{\mu_N}{\sum_{k=1}^N\mu_k}.
\end{equation*}
In particular this means that $s_{N-1} \sim t_{N}$. Moreover,
\begin{equation*}
\frac{1}{w_{N-1}} = \sum_{k=1}^{N}\frac{\mu_k}{(s_{N-1}-t_k)^2} \sim \frac{\mu_N}{(t_N-s_{N-1})^2} \sim \frac{1}{t_N^2}.
\end{equation*}
So, $w_{N-1} \to \infty$. By \eqref{coeffFormulas} we see that $q_2 \sim  s_{N-1} \sim t_N$. This means that ratio $q_2$ over $t_{N-1}$ can be arbitrarily big.

\section{Proof of Theorem \ref{FTTheorem} and Fock-Type Spaces}
\label{SectProofFT}
We use notation from Section \ref{IntroFTSpaces}.
We also assume that $r_{k+1}>\lambda r_k$, $\data_{k+1} > \low \data_k$ and $\frac{\data_{k+1}}{r_{k+1}^2}<\theta\frac{\data_k}{r_k^2}$ for some $\lambda > 1000, \low >1000$ and $\theta<1$.
%
%
The corresponding spectral data measure of Jacobi matrix is given by 
$\spd = \sum_{k=1}^{\infty}\spd_k\delta_{r_k},$
where 
\begin{equation} \label{FTspd}
\spd_k = c \nu_{k}^{-1}r_k^2\prod_{l\neq k}\left(1-\dfrac{r_k}{r_l}\right)^{-2}.
\end{equation} 
Here $c$ is a normalizing constant such that $\spd(\R) = 1.$
Clearly, 
\begin{equation} \label{FTmainIneq}
\left(1-\frac{10}{\lambda}\right) \low \left(\frac{r_{k+1}}{r_k}\right)^{2k-2}
<\frac{\spd_{k}}{\spd_{k+1}}
< \left(1+\frac{10}{\lambda}\right)\theta \left(\frac{r_{k+1}}{r_k}\right)^{2k}.
\end{equation}

Now our goal is to estimate spectral data $\left\{\spd_k^{(1)}, r_k^{(1)}\right\}$ after one step of the Stieltjes algorithm.

\subsection{Change of variables}
Let $\tau_k = r_k^{-1}$, $\alpha_k = \spd_k r_k^{-1}$ and $\zeta = z^{-1}$. Note that
\begin{equation*}
\sum_{k\geq 1}\frac{\spd_k}{r_k-z} = z^{-1}\sum_{k\geq 1}\frac{\spd_k r_k^{-1}}{z^{-1}-r_k^{-1}} = \zeta \sum_{k\geq 1} \frac{\alpha_k}{\zeta-\tau_k},
\end{equation*}
\begin{equation*}
\sum_{k\geq 1} \frac{\spd_k}{(r_k-z)^2} = \frac{1}{r_k z^2}\sum_{k\geq 1} \frac{\spd_k r_k^{-1}}{(z^{-1}-r_k^{-1})^2} =  \zeta^2 \sum_{k\geq 1} \frac{\alpha_k \tau_k}{(\zeta-\tau_k)^2}.
\end{equation*}

Thus, analyzing one step of Stieltjes algorithm for $\{\alpha_k, \tau_k\}$ will help us to analyze Stieltjes algorithm for $\{\spd_k, r_k\}$. Let $\{\alpha_k^{(1)}, \tau_k^{(1)}\}$ be spectral data after one step of Stieltjes algorithm, applied to data $\{\alpha_k, \tau_k\}$.


Then $\tau_k > \lambda \tau_{k+1}$ and
\begin{equation*}
\left(1-\frac{10}{\lambda}\right)\low \alpha_{k+1} \left(\frac{\tau_k}{\tau_{k+1}}\right)^{2k-1}
<\alpha_k 
< \left(1+\frac{10}{\lambda}\right) \theta \alpha_{k+1}  \left(\frac{\tau_k}{\tau_{k+1}}\right)^{2k+1}.
\end{equation*}

For the reader convenience in this subsection we will consider the finite case (taking $\{\alpha_{k}, \tau_k\}_{k=1}^N$). Then we will take limit in $N$.

We put $\mu_{k} = \alpha_{N-k+1}$, $t_k = \tau_{N-k+1}$. And now we can proceed in the same manner as in Section \ref{ProofLemmas}, since $\frac{\mu_{n+1}}{\mu_n}>\left(1-\frac{10}{\lambda}\right)\low \left(\frac{t_{n+1}}{t_n}\right)^{2N-2n-1}$. In particular, Statements \ref{approx}, \ref{lacun}, Lemma \ref{bound} and \eqref{slowgrowth} still hold 
(we again use \eqref{root}).

Now, $r^{(1)}_n = \dfrac{1}{\tau^{(1)}_{N-n}}$ which is approximated by $\dfrac{1}{s_{N-n}}$ and 
$$\spd_n^{(1)} = \left(\sum_{k\geq 1} \frac{\spd_k}{(r_n^{(1)}-r_k)^2}\right)^{-1}
=\left((\tau_n^{(1)})^3\sum_{k\geq 1} \frac{\alpha_k}{(\tau_n^{(1)}-\tau_k)^2}\right)^{-1},$$
which is approximated by $\dfrac{w_{N-n}}{s_{N-n}^3}$.

\begin{lmm}
\label{FTlowstepweak}
Let  $\delta_n = 4\max\left(\frac{N-n-1}{\lambda^{N-n-1}}, \frac{N-n-1}{\low^{N-n-2}\lambda}\right)$.
Then $(1+\delta_n)^2\frac{w_{n+1}}{w_n}> \frac{\mu_{n+1}}{\mu_n}$.
\end{lmm}

\begin{proof}

Since $\frac{\mu_{n+1}}{(t_{n+1}-s_{n})^2}>\frac{\mu_{n+1}}{\mu_{n}}\cdot\frac{\mu_{n}}{(s_{n+1}-t_{n})^2}$,
it is sufficient to prove that  
\begin{multline}\label{FTlsweakmain}
\left(\frac{(1+\delta_n)\mu_n}{s_n-t_n}-\dfrac{\mu_{n+1}}{s_{n+1}-t_{n+1}}\right)
\left(\frac{(1+\delta_n)\mu_n}{s_n-t_n}+\dfrac{\mu_{n+1}}{s_{n+1}-t_{n+1}}\right)\\
>\smm_{k \neq n, n+1} \dfrac{\mu_{n+1}\mu_k}{(s_{n+1}-t_k)^2}.
\end{multline}

Now, from \eqref{keyineq} we have
\begin{equation*}
\dfrac{\mu_{n+1}}{s_{n+1}-t_{n+1}}<
\dfrac{\mu_n}{s_n-t_n}
+\dfrac{\sum_{k=1}^{n-1}\mu_k(t_{n+1}-t_n)}{(t_{n+1}-t_{n-1})(t_n-t_{n-1})}+\smm_{l=n+2}^N \dfrac{\mu_l(t_{n+1}-t_n)}{(t_l-s_{n+1})^2}.
\end{equation*}


Note that
\begin{gather*}
\dfrac{\sum_{k=1}^{n-1}\mu_k(t_{n+1}-t_n)}{(t_{n+1}-t_{n-1})(t_n-t_{n-1})}
<\left(\dfrac{\lambda}{\lambda-1}\right)^3 \cdot\dfrac{\mu_{n-1}}{t_n}
<\left(\dfrac{\lambda}{\lambda-1}\right)^3 \cdot \frac{\mu_N}{\lambda \low^{N-n} t_{N}}.
\end{gather*}
%
Moreover,  $\smm_{l=n+2}^N \dfrac{\mu_l(t_{n+1}-t_n)}{(t_l-s_{n+1})^2}<\left(\frac{\lambda}{\lambda-2}\right)^2 \sum_{l=n+2}^{N}\frac{\mu_l t_{n+1}}{t_l^2}.$
Notice that
\begin{equation*}
\sum_{l=n+2}^{N}\frac{\mu_l t_{n+1}}{t_l^2}<\sum_{l=n+2}^{N}\frac{\mu_l}{t_l \lambda^{l-n-1}}
<\frac{\mu_N}{t_N} \max\left(\dfrac{N-n-1}{\lambda^{N-n-1}}, \dfrac{N-n-1}{\low^{N-n-2}\lambda}\right).
\end{equation*}

Then
\begin{gather*}
\frac{(1+\delta_n)\mu_n}{s_n-t_n}-\dfrac{\mu_{n+1}}{s_{n+1}-t_{n+1}}>\frac{\delta_n\mu_n}{s_n-t_n}-\frac{\delta_n}{2}\frac{\mu_N}{t_N}.
\end{gather*}
From Statement \ref{approx}, $\frac{\mu_n}{s_n-t_n}>\frac{\mu_N}{t_N}.$
Then we get
$$\frac{(1+\delta_n)\mu_n}{s_n-t_n}-\dfrac{\mu_{n+1}}{s_{n+1}-t_{n+1}}
>\dfrac{\sum_{k=1}^{n-1}\mu_k(t_{n+1}-t_n)}{(t_{n+1}-t_{n-1})(t_n-t_{n-1})}+\smm_{l=n+2}^N \dfrac{\mu_l(t_{n+1}-t_n)}{(t_l-s_{n+1})^2}.$$
Furthermore, $\frac{(1+\delta_n)\mu_n}{s_n-t_n}+\frac{\mu_{n+1}}{s_{n+1}-t_{n+1}}>\frac{\mu_{n+1}}{t_{n+1}-t_{n}}.$ These two inequalities imply \eqref{FTlsweakmain}.
\end{proof}

\begin{lmm}
\label{FTupstep}
Let $\loctheta = \frac{\mu_{n+1}t_{n}^2}{\mu_{n}t_{n+1}^2}$.
Then $\frac{w_{n+1}s_{n}^2}{w_{n}s_{n+1}^2}<\left(1+\eps_n\right)^3 \loctheta$, where $\eps_n = 5(\low-1)^{1+n-N}$.
\end{lmm}
\begin{proof}
Note that $\loctheta > \lambda$, since $\frac{\mu_{n+1}}{\mu_n}>\frac{t_{n+1}^3}{t_n^3}$ for $n\leq N-2$.

We want to prove the following inequality
\begin{equation}
\label{FTgoallessstrong}
\smm_{k=1}^N \dfrac{\mu_ks_n^2}{(t_k-s_n)^2}<\left(1+\eps_n\right)^3\loctheta\smm_{k=1}^N \dfrac{\mu_ks_{n+1}^2}{(t_k-s_{n+1})^2}.
\end{equation}

\textbf{Step 1.} First, we compare the main terms.
We are going to show that
\begin{equation}
\label{FTmainlessstrong}
\dfrac{\mu_n s_n^2}{(s_n-t_n)^2}<\left(1+\eps_n\right)^2\loctheta \dfrac{\mu_{n+1}s_{n+1}^2}{(s_{n+1}-t_{n+1})^2}.
\end{equation}
It is sufficient to prove that
\begin{equation*}
\dfrac{\mu_n^2 s_n^2}{t_n^2(s_n-t_n)^2}<\left(1+\eps_n\right)^2\dfrac{\mu_{n+1}^2s_{n+1}^2}{t_{n+1}^2(s_{n+1}-t_{n+1})^2},
\end{equation*}
which will follow from
\begin{equation*}
\dfrac{\mu_n}{s_n-t_n}-\dfrac{\mu_{n+1}}{s_{n+1}-t_{n+1}}<\dfrac{\mu_{n+1}}{t_{n+1}}-\dfrac{\mu_n}{t_n}+\eps_n\cdot\dfrac{\mu_{n+1}}{s_{n+1}-t_{n+1}}.
\end{equation*}
We use \eqref{upperRootDiffer}.
Let us prove that the right-hand side in \eqref{upperRootDiffer} is smaller than
$\frac{\mu_{n+1}}{t_{n+1}}-\frac{\mu_n}{t_n}+\eps_n\cdot\frac{\mu_{n+1}}{s_{n+1}-t_{n+1}}$. So, we want to show that 
\begin{equation*}
\frac{\mu_ns_{n+1}}{(s_{n+1}-t_n)t_n}+\frac{\mu_{n+1}s_n}{(t_{n+1}-s_n)t_{n+1}}
<\eps_n\cdot\dfrac{\mu_{n+1}}{s_{n+1}-t_{n+1}}+\smm_{k\neq n, n+1}\frac{\mu_k(s_{n+1}-s_{n})}{(s_{n+1}-t_k)(s_n-t_k)}.
\end{equation*}
From \eqref{slowgrowth} we see that both terms on the left-hand side are less than $\frac{\eps_n}{2}\cdot\frac{\mu_{n+1}}{s_{n+1}-t_{n+1}}$.
Therefore, we have proved inequality \eqref{FTmainlessstrong}.  
\vspace{1pt}

\textbf{Step 2.} In order to prove \eqref{FTgoallessstrong} it is sufficient to estimate $\sum_{k\neq n} \frac{\mu_ks_n^2}{(t_k-s_n)^2}$.

We have
\begin{equation*}
\dfrac{\mu_{n+1} s_n^2}{(t_{n+1}-s_n)^2}
<\loctheta\eps_n\dfrac{\mu_{n+1}s_{n+1}^2}{(s_{n+1}-t_{n+1})^2};
\end{equation*}
\begin{equation*}
\dfrac{\mu_ks_n^2}{(t_k-s_n)^2}
<\loctheta\dfrac{\mu_{k} s_{n+1}^2}{(t_{k}-s_{n+1})^2}, \qquad k \geq n+2.
\end{equation*}

%
Also,
\begin{equation*}
\smm_{k=1}^{n-1} \dfrac{\mu_ks_n^2}{(t_k-s_n)^2}
<\frac{\mu_{n-1}}{(1-\frac{1}{\lambda})^3}
<\loctheta\frac{\mu_{n}s_{n+1}^2}{(s_{n+1}-t_n)^2}.
\end{equation*}
This proves \eqref{FTgoallessstrong}.
\end{proof}

\begin{crl} \label{FTmainCrl}
We have $r_{n+1}^{(1)}>\lambda r_{n}^{(1)}$ and 
$$\left(1-\frac{20n}{\low^n}\right)\frac{\spd_{n+1} r_{n+2}^2 }{\spd_{n+2} r_{n+1}^2}
<\frac{\spd_{n}^{(1)}}{\spd_{n+1}^{(1)}}
< \left(1+\frac{20}{\low^n}\right) \frac{\spd_{n+1}r_{n+2}^2}{\spd_{n+2}r_{n+1}^2}.$$
\end{crl}
\begin{proof}
From Statement \ref{lacun} we know that $s_{k+1}>\lambda s_k$, hence, since $s_{N-n}^{-1}$ approximate $r_n^{(1)}$, we get that $r_{n+1}^{(1)}>\lambda r_{n}^{(1)}$.

In turn, $\spd_{n}^{(1)}$ is approximated by $\dfrac{w_{N-n}}{s_{N-n}^3}$. Moreover,
\begin{equation*}
\dfrac{w_{N-n}s_{N-n-1}^3}{w_{N-n-1}s_{N-n}^3}
>(1-4\delta_{N-n-1})^2 \frac{\mu_{N-n}s_{N-n-1}^3}{\mu_{N-n-1}s_{N-n}^3}
> \left(1-\frac{20n}{\low^n}\right)\frac{\spd_{n+1} r_{n+2}^2 }{\spd_{n+2} r_{n+1}^2};
\end{equation*}
\begin{equation*}
\dfrac{w_{N-n}s_{N-n-1}^3}{w_{N-n-1}s_{N-n}^3}
< (1+\varepsilon_{N-n-1})^3 \frac{\mu_{N-n}t_{N-n-1}^2 s_{N-n-1}}{\mu_{N-n-1}t_{N-n}^2s_{N-n}} 
< \left(1+\frac{20}{\low^n}\right) \frac{\spd_{n+1}r_{n+2}^2}{\spd_{n+2}r_{n+1}^2}.
\end{equation*}
Now we just take $N\to \infty$.
\end{proof}
Now combine corollary above with \eqref{FTmainIneq}.

\begin{crl} We have
\begin{equation*}
\left(1-\frac{20n^2}{\low^n}\right) \low \left(\frac{r_{n+1}^{(1)}}{r_n^{(1)}}\right)^{2n-2}
<\frac{\spd_{n}^{(1)}}{\spd_{n+1}^{(1)}}
<  \left(1+\frac{20n}{\low^n}\right)\theta \left(\frac{r_{n+1}^{(1)}}{r_n^{(1)}}\right)^{2n}.
\end{equation*}
\end{crl}

\begin{crl} For evety $l>0$,
\begin{equation*}
\frac{1}{10} \low \left(\frac{r_{n+1}^{(l)}}{r_n^{(l)}}\right)^{2n-2}
<\frac{\spd_{n}^{(l)}}{\spd_{n+1}^{(l)}}
<  10 \theta \left(\frac{r_{n+1}^{(l)}}{r_n^{(l)}}\right)^{2n}.
\end{equation*}
\end{crl}

\subsection{Proof of Theorem \ref{FTTheorem}}
Let $\{\rho_n, q_n\}$ be the entries of the corresponding Jacobi matrix. We write  $\spd_k^{(n)}, r_k^{(n)}$ for the data on the $n$-th step of the Stieltjes algorithm. Then

\begin{equation}
\label{FTJacobiEntries}
q_n= \sum_{k\geq 1}\spd_k^{(n)} r_k^{(n)}, \qquad
\rho_n^2 = \sum_{k\geq 1}\spd_k^{(n)} \left( r_k^{(n)}\right)^2  - \left(\sum_{k\geq 1}\spd_k^{(n)}  r_k^{(n)} \right)^2.
\end{equation}

Using  \eqref{FTspd},  \eqref{FTJacobiEntries} and Corollary \ref{FTmainCrl} we get the desired inequalities.

\section{Proof of Canonical Systems Theorems}
\label{ProofCanonTheorems}
\subsection{Proof of Theorem \ref{ThrCanon}}
Because of \eqref{JacCanonDiagFirst}, we have
$\dfrac{\sqrt{1-\scprod_1^2}}{\scprod_1}=q_1l_1 = 1000.$
Since
\begin{equation}
\frac{1}{\scprod}-1<\frac{\sqrt{1-\scprod^2}}{\scprod}<\frac{1}{\scprod},
\end{equation}
We have $\frac{1}{1001}<\scprod_1<\frac{1}{1000}$.

\smallskip

Now we prove Theorem \ref{ThrCanon} by induction on $n$.

\textbf{Induction step.} 
$n\mapsto n+1.$

\textit{Estimates of $l_{n+1}$.}
Because of the induction hypothesis we have
$\frac{1}{\scprod_n}-1<q_nl_n<\frac{2}{\scprod_n}.$
Hence,
$\frac{1-\scprod_n}{q_n}<l_n\scprod_n<\frac{2}{q_n}.$
Using \eqref{JacCanonNonDiag}, we get
\begin{equation}\label{lenIneq}
\frac{q_n^2}{2\rho_n^2}<\frac{l_{n+1}}{l_n}<\frac{q_n^2}{(1-\scprod_n)^2\rho_n^2}.
\end{equation}

\textit{Estimates of $\scprod_{n+1}$.}
Clearly,
$\frac{\sqrt{1-\scprod_n^2}}{\scprod_n}<q_nl_n<2\frac{\sqrt{1-\scprod_n^2}}{\scprod_n}.$
Therefore,
\begin{equation*}
\frac{1}{2}\cdot\frac{\sqrt{1-\scprod_{n+1}^2}}{\scprod_{n+1}}\cdot\frac{\scprod_n}{\sqrt{1-\scprod_n^2}}
<\frac{q_{n+1}l_{n+1}}{q_n l_n}
<\frac{\sqrt{1-\scprod_{n+1}^2}}{\scprod_{n+1}}\cdot\frac{\scprod_n}{\sqrt{1-\scprod_n^2}}+1.
\end{equation*}
Using \eqref{lenIneq},
\begin{equation}
\label{ProofCanSyst2}
\frac{q_{n+1}q_n}{3\rho_n^2}
<\frac{\scprod_{n}}{\scprod_{n+1}}
<\frac{3q_{n+1}q_n}{\rho_n^2}.
\end{equation}
Now combination of \eqref{lenIneq} and \eqref{ProofCanSyst2} gives the desired inequalities.
\begin{flushright}
$\Box$
\end{flushright}

\subsection{Proof of Theorem \ref{ThrRevearsedCanon}}
This proof is analogous to the proof of Theorem \ref{ThrCanon}. We again proof \eqref{lenIneq} and \eqref{ProofCanSyst2} and deduce Theorem \ref{ThrRevearsedCanon} from them.
\begin{flushright}
$\Box$
\end{flushright}



\end{document}